\documentclass[12pt]{amsart}
\usepackage{amsmath}
\usepackage{amstext}
\usepackage{amsfonts}
\usepackage{amssymb}
\usepackage{amsthm}
\usepackage{amsrefs}

\usepackage{microtype}
\usepackage{hyperref}

\usepackage{tikz}
\usetikzlibrary{arrows}
\usetikzlibrary{positioning}

\theoremstyle{plain}
\newtheorem{thm}{Theorem}[section]
\newtheorem{conj}[thm]{Conjecture}
\newtheorem{cor}[thm]{Corollary}
\newtheorem{prop}[thm]{Proposition}
\newtheorem{lem}[thm]{Lemma}

\theoremstyle{definition}
\newtheorem{defn}[thm]{Definition}

\newtheorem{rem}[thm]{Remark}

\numberwithin{equation}{section}

\newcommand{\bC}{{\mathbb{C}}}
\newcommand{\bF}{{\mathbb{F}}}

\newcommand{\A}{{\mathcal{A}}}
\newcommand{\B}{{\mathcal{B}}}

\newcommand{\J}{{\mathcal{J}}}

\newcommand{\M}{{\mathcal{M}}}

\renewcommand{\O}{{\mathcal{O}}}
\renewcommand{\P}{{\mathcal{P}}}
\renewcommand{\S}{{\mathcal{S}}}
\newcommand{\T}{{\mathcal{T}}}
\newcommand{\U}{{\mathcal{U}}}

\renewcommand{\phi}{\varphi}

\newcommand{\rcp}[2]{{#1} \rtimes_\mathrm{r} {#2}}

\newcommand{\ca}{\mathrm{C}^*}

\newcommand{\fb}{\partial_F G}
\newcommand{\hb}{\partial_H G}

\begin{document}
\title[Boundaries of reduced C*-algebras of discrete groups]{Boundaries of reduced C*-algebras of discrete groups}

\author[M. Kalantar]{Mehrdad Kalantar}
\address{Department of Mathematics and Statistics\\ Carleton University\\
Ottawa, ON \; K1S 5B6 \\Canada}
\email{mkalanta@math.carleton.ca}

\author[M. Kennedy]{Matthew Kennedy}
\address{Department of Mathematics and Statistics\\ Carleton University\\
Ottawa, ON \; K1S 5B6 \\Canada}
\email{mkennedy@math.carleton.ca}

\begin{abstract}
For a discrete group $G$, we consider the minimal C*-subalgebra of $\ell^\infty(G)$ that arises as the image of a unital positive $G$-equivariant projection. This algebra always exists and is unique up to isomorphism. It is trivial if and only if $G$ is amenable. We prove that, more generally, it can be identified with the algebra $C(\fb)$ of continuous functions on Furstenberg's universal $G$-boundary $\fb$.

This operator-algebraic construction of the Furstenberg boundary has a number of interesting consequences. We prove that $G$ is exact precisely when the $G$-action on $\fb$ is amenable, and use this fact to prove Ozawa's conjecture that if $G$ is exact, then there is an embedding of the reduced C*-algebra $\ca_r(G)$ of $G$ into a nuclear C*-algebra which is contained in the injective envelope of $\ca_r(G)$.

The algebra $C(\fb)$ arises as an injective envelope in the sense of Hamana, which implies rigidity results for certain $G$-equivariant maps. We prove a generalization of a rigidity result of Ozawa for $G$-equivariant maps between spaces of functions on the hyperbolic boundary of a hyperbolic group. Our result applies to hyperbolic groups, but also to groups that are not hyperbolic or even relatively hyperbolic, including certain mapping class groups.

It is a longstanding open problem to determine which groups are C*-simple, in the sense that the algebra $\ca_r(G)$ is simple. We prove that this problem can be reformulated as a problem about the structure of the $G$-action on the Furstenberg boundary. Specifically, we prove that a discrete group $G$ is C*-simple if and only if the $G$-action on the Furstenberg boundary is topologically free. We apply this result to prove that Tarski monster groups are C*-simple. This provides another solution to a problem of de la Harpe (recently answered by Olshanskii and Osin) about the existence of C*-simple groups with no free subgroups.
\end{abstract}

\thanks{Second author partially supported by research grant from NSERC (Canada).}
\maketitle

\section{Introduction}

A discrete group $G$ is said to be \emph{amenable} if there is a left-invariant mean on the algebra $\ell^\infty(G)$, i.e. a unital positive linear map $\lambda : \ell^\infty(G) \to \bC$ that is invariant with respect to the action of $G$ on $\ell^\infty(G)$ by left translation. By identifying the algebra $\bC$ with a subalgebra of $\ell^\infty(G)$, the map $\lambda$ can be viewed as a unital positive $G$-equivariant projection. From this perspective, a group $G$ is non-amenable if the algebra $\bC$ is ``too small,'' relative to $G$, to admit such a projection.

In this paper, we consider a natural generalization of the notion of amenability. Specifically, for an arbitrary discrete group $G$, we consider the minimal C*-subalgebra of $\ell^\infty(G)$ that arises as the image of a positive $G$-equivariant projection on $\ell^\infty(G)$. It follows from Hamana's work on $G$-injective envelopes \cites{Ham78,Ham85} that this algebra exists and is unique up to isomorphism.  We prove that it can be identified with the algebra $C(\fb)$ of continuous functions on Furstenberg's universal $G$-boundary $\fb$.

This operator-algebraic viewpoint provides a powerful tool for the study of Furstenberg boundaries of discrete groups which, generally speaking, are not as well-studied as their counterparts for connected Lie groups. Using operator-algebraic techniques, we obtain a number of results about the structure of the Furstenberg boundary of a discrete group that appear to be new. On the other hand, we also show that the $G$-action on the Furstenberg boundary encodes a great deal of information about the operator-algebraic and ergodic-theoretic properties of $G$.

We recall that a discrete group $G$ is \emph{exact} if the reduced C*-algebra $\ca_r(G)$ of $G$ is exact. This notion, introduced by Kirchberg and Wasserman in \cite{KW99} has several equivalent formulations. Ozawa proved in  \cite{Oza00}*{Theorem 3} that $G$ is exact if and only the $G$-action on the Stone-Cech compactification $\beta G$ of $G$ is amenable. We prove the exactness of $G$ can also be detected by the amenability of the $G$-action on the Furstenberg boundary $\fb$.

\begin{thm} \label{main-thm:exact-iff-amenable}
Let $G$ be a discrete group. Then $G$ is exact if and only if the $G$-action on the Furstenberg boundary $\fb$ is amenable.
\end{thm}

Wasserman \cite{Was90} proved that a C*-algebra is exact if it can be embedded into a nuclear C*-algebra, and Kirchberg \cite{Kir95} proved the converse, that every exact C*-algebra can be embedded into a nuclear C*-algebra. A deep result of Kirchberg and Phillips \cite{KP00} implies that for a separable exact C*-algebra, the nuclear C*-algebra can be taken to be the Cuntz algebra on two generators.

Although concrete nuclear embeddings have been constructed for exact C*-algebras in certain special cases (see e.g. \cites{Cho79,KN11}), in general the nuclear embeddings that are guaranteed to exist by the more abstract results above can be difficult to realize. Additionally, the image of an exact C*-algebra under a nuclear embedding can be relatively small, and consequently, properties like simplicity and primeness are not necessarily reflected in the larger C*-algebra.

Ozawa \cite{Oza07} considered the problem of constructing a more ``tight'' nuclear embedding of the reduced C*-algebra $\ca_r(\bF_n)$ of the free group $\bF_n$ on $n$ generators. This C*-algebra is known to be exact. Ozawa proved there is a canonical nuclear C*-algebra $N(\ca_r(\bF_n))$ such that 
\[
\ca_r(\bF_n) \subset N(\ca_r(\bF_n)) \subset I(\ca_r(\bF_n)),
\]
where $I(\ca_r(\bF_n))$ denotes the injective envelope of $\ca_r(\bF_n)$.

More generally, Ozawa conjectures it should be possible to construct such an embedding for any exact C*-algebra.

\begin{conj}[Ozawa's Conjecture] \label{conj:ozawa}
Let $\A$ be an exact C*-algebra. Then there is a nuclear C*-algebra $N(\A)$ such that
\[
\A \subset N(\A) \subset I(\A),
\]
where $I(\A)$ denotes the injective envelope of $\A$.
\end{conj}

We prove Conjecture \ref{conj:ozawa} for the reduced C*-algebra of every discrete exact group.

\begin{thm} \label{main-thm:exact-nuclear-embedding}
Let $G$ be a discrete exact group, and let $\ca_r(G)$ denote the reduced C*-algebra of $G$. There is a canonical unital nuclear C*-algebra $N(\ca_r(G))$ such that
\[
\ca_r(G) \subset N(\ca_r(G)) \subset I(\ca_r(G)),
\]
where $I(\ca_r(G))$ denotes the injective envelope of $\ca_r(G)$. The algebra $N(\ca_r(G))$ is simple if $\ca_r(G)$ is simple, and prime if and only if $\ca_r(G)$ is prime.
\end{thm}

Ozawa takes $N(\ca_r(\bF_n))$ to be the reduced crossed product $\rcp{C(\partial \bF_n)}{\bF_n}$, where $\partial \bF_n$ denotes the hyperbolic boundary of $\bF_n$. It is known that the $\bF_n$-action on $\partial \bF_n$ is amenable, and hence that this crossed product is nuclear (see e.g. \cite{BroOza08}*{Theorem 4.3.4}).

For an arbitrary, potentially non-hyperbolic, exact group, Theorem \ref{main-thm:exact-iff-amenable} suggests that the role of the hyperbolic boundary in Ozawa's proof should be played here by the Furstenberg boundary. We take $N(\ca_r(G))$ to be the reduced crossed product $\rcp{C(\fb)}{G}$.

We construct the algebra $C(\fb)$ of continuous functions on $\fb$ as a $G$-injective envelope in the sense of Hamana \cites{Ham78,Ham85}. This construction of the Furstenberg boundary implies some powerful rigidity results for $G$-equivariant maps on $C(\fb)$, and hence by contravariance, for $G$-equivariant maps on $\fb$. However, even the fact that $C(\fb)$ is an injective algebra appears to have gone unnoticed. We make use of these properties throughout our paper.

A key ingredient in Ozawa's paper is a rigidity result \cite{Oza07}*{Proposition 3} for unital positive equivariant maps between spaces of functions on the hyperbolic boundary of a hyperbolic group. We prove a generalization of this result which imposes slightly weaker requirements.

\begin{thm} \label{main-thm:rigidity-maps}
Let $G$ be a non-amenable hyperbolic group and let $\mu$ be an irreducible probability measure on $G$. Let $\nu$ be a $\mu$-stationary measure on the hyperbolic boundary $\partial G$. If $\phi : C(\partial G) \to L^\infty(\partial G, \nu)$ is a unital positive $G$-equivariant map, then $\phi = \operatorname{id}$.
\end{thm}

The proof of our result is significantly different than Ozawa's. In fact, our techniques also apply to groups that are not necessarily hyperbolic or even relatively hyperbolic, including certain mapping class groups. We utilize Jaworski's  \cite{Jaworski} theory of strongly approximately transitive measures, combined with uniqueness results of Kaimanovich \cite{Kai00} and Kaimanovich-Masur \cite{Kai-Mas96} for stationary measures.

A group $G$ is \emph{C*-simple} if the reduced C*-algebra $\ca_r(G)$ is simple. A great deal of work has been devoted to understanding which groups are C*-simple. It turns out this problem is equivalent to the topological freeness of the $G$-action on the Furstenberg boundary. We prove the following result.

\begin{thm}
Let $G$ be a discrete group, and let $\fb$ denote the Furstenberg boundary of $G$. Then the following are equivalent:
\begin{enumerate}
\item The group $G$ is C*-simple, i.e. the reduced C*-algebra $\ca_r(G)$ is simple.
\item The reduced crossed product $\rcp{C(\fb)}{G}$ is simple.
\item The reduced crossed product $\rcp{C(B)}{G}$ is simple for some $G$-boundary $B$.
\item The $G$-action on $\fb$ is topologically free.
\item The $G$-action on some $G$-boundary is topologically free.
\end{enumerate}
\end{thm}

We apply this result to prove that Tarski monster groups are C*-simple. This provides another solution to a problem of de la Harpe (recently answered by Olshanskii and Osin) about the existence of C*-simple groups with no free subgroups.

\begin{thm}
Tarski monster groups are C*-simple.
\end{thm}

In addition to this introduction, this paper has five other sections. In Section \ref{sec:prelim} we briefly review the requisite background material. In Section \ref{sec:op-alg-construction}, we construct the Furstenberg boundary using Hamana's theory of equivariant injective envelopes, and study some of its properties. In Section \ref{sec:exactness-and-nuclear-embeddings} we prove a group is exact if and only its action on the Furstenberg boundary is amenable, and we use this to prove our nuclear embedding theorem for the reduced C*-algebra of an exact group.  In Section \ref{sec:rigidity-stationary-measures} we prove our rigidity results. In Section \ref{sec:C*-sim} we consider the connection between the C*-simplicity of a group and the ergodic properties of its action on the Furstenberg boundary

We note that our construction of the Furstenberg boundary generalizes to certain locally compact quantum groups. In fact, many of the results in this paper also hold in that setting. This suggests that our construction provides an appropriate quantum-group-theoretic analogue of the Furstenberg boundary. We intend to pursue these issues in a later paper.

\subsection*{Acknowledgements}
We are grateful to Zhong-Jin Ruan for introducing us to Ozawa's conjecture, and for his support and encouragement throughout this project. We are also grateful to our colleagues Wojciech Jaworski and Vadim Kaimanovich for many helpful comments and suggestions. We extend our sincere thanks to R\'emi Boutonnet, Masamichi Hamana and Nicolas Monod for their feedback on an earlier draft of this paper. Finally, we thank Narutaka Ozawa for many helpful suggestions, and in particular for showing us how to strengthen Theorem \ref{thm:c-star-simplicity}.

\section{Preliminaries} \label{sec:prelim}

\subsection{Group actions}

Let $G$ be a discrete group.  A locally compact Hausdorff space $X$ is said to be a \emph{$G$-space} if there is a continuous homomorphism from $G$ into the group of homeomorphisms on $X$ that sends the the identity element of $G$ to the identity map on $X$. We refer to such a homomorphism as a \emph{$G$-action} on $X$.

If $X$ is a compact $G$-space, then by contravariance, the $G$-action on $X$ induces a $G$-action on the C*-algebra $C(X)$ of continuous functions on $X$,
\[
s f(x) = f(s^{-1} x), \quad s \in G,\ f \in C(X),\ x \in X.
\]
Similarly, the $G$-action on $X$ induces a $G$-action on the set $\P(X)$ of probability measures on $X$,
\[
s \nu(Y) = \nu(s^{-1}Y), \quad s \in G,\ \nu \in \P(X),\ Y \subset X.
\]
Equipped with this $G$-action and endowed with the weak* topology, the space $\P(X)$ is a compact $G$-space.

If $X$ is a compact $G$-space, then for every probability measure $\nu \in \P(X)$, there is a natural unital positive $G$-equivariant map $P_\nu : C(X) \to \ell^\infty(G)$,
\begin{equation}\label{eqn:poisson-map}
P_\nu(f)(s) = \langle f, s\nu \rangle = \int_{X} f(sx) \, d\nu(x), \quad s \in G,\ f \in C(X).
\end{equation}
We will refer to the map $P_\nu$ as the \emph{Poisson map} corresponding to $\nu$.

\subsection{Injective envelopes} \label{sec:G-inj-envelope}

We require Hamana's theory  \cite{Ham85} of injective envelopes for operator systems equipped with a group action. For the general theory of operator systems and injective envelopes, we refer the reader to Hamana's papers \cites{Ham79a,Ham79b}, or to Paulsen's book \cite{Pau02}*{Chapter 15}.

Let $G$ be a discrete group, and let $\S$ be an \emph{operator system}, i.e. a unital self-adjoint subspace of a unital C*-algebra. We say $\S$ is a $G$-operator system if there is a homomorphism from  $G$ into the group of order isomorphisms on $\S$ that sends the the identity element of $G$ to the unit in $\S$. We refer to such a homomorphism as a \emph{$G$-action} on $\S$.

If $\T$ is another $G$-operator system, then a map $\phi : \S \to \T$ is said to be \emph{$G$-equivariant} if it commutes with the $G$-actions on $\S$ and $\T$,
\[
\phi(s T) = s \phi(T),\quad \forall s \in G,\ \forall T \in \S.
\]

A $G$-operator system $\U$ is said to be \emph{$G$-injective} if for every unital completely isometric $G$-equivariant map $\iota : \S \to \T$ and every unital completely positive $G$-equivariant map $\psi : \S \to \U$ there is a unital completely positive $G$-equivariant map $\hat{\phi} : \T \to \U$ such that $\hat{\phi} \iota = \phi$.

Hamana proves in \cite{Ham85}*{Lemma 2.2} that if $\S$ is an injective operator system, then the $G$-operator system $\ell^\infty(G,\S)$ is always $G$-injective. 

A \emph{$G$-extension of $\S$} is a pair $(\T, \iota)$ consisting of a $G$-operator system $\T$, and a completely isometric $G$-equivariant map $\iota : \S \to \T$. We say the image $\iota(\S)$ of $\S$ under $\iota$ is a \emph{completely isometric $G$-equivariant copy of $\S$}.

A $G$-extension $(\U, \iota)$ is \emph{$G$-injective} if $\U$ is $G$-injective. It is \emph{$G$-essential} if for every unital completely positive $G$-equivariant map $\phi : \U \to \T$ such that $\phi \iota$ is completely isometric on $\S$, $\phi$ is necessarily completely isometric on $\U$.  It is \emph{$G$-rigid} if for every unital completely positive $G$-equivariant map $\phi : \U \to \U$ such that $\phi \iota = \iota$ on $\S$, $\phi$ is necessarily the identity map on $\U$.

\begin{defn}
Let $G$ be a discrete group, and let $\S$ be a $G$-operator system. A $G$-extension of $\S$ that is $G$-injective and $G$-essential is said to be a \emph{$G$-injective envelope of $\S$}.
\end{defn}

We note that by \cite{Ham85}*{Lemma 2.4}, every $G$-injective envelope of $\S$ is $G$-rigid. 

\begin{thm}[Hamana]\label{thm:existence-G-inj-envelope}
Let $G$ be a discrete group, and let $\S$ be a $G$-operator system. Then $\S$ has a $G$-injective envelope $(I_G(\S), \kappa)$. This injective envelope is unique, in the sense that for every $G$-injective envelope $(\U, \iota)$ of $\S$, there is a completely isometric $G$-equivariant map $\phi : I_G(\S) \to \U$ such that $\phi \kappa = \iota$.
\end{thm}

If $\S$ is a $G$-operator system, there is a natural unital completely isometric $G$-equivariant map $\iota : \S \to \ell^\infty(G,\S)$ embedding $\S$ into $\ell^\infty(G,\S)$,
\[
\iota(S)(s) = s^{-1} S, \quad S \in \S,\ s \in G.
\]

If $\A$ is an injective C*-algebra, then $\ell^\infty(G,\A)$ is $G$-injective and hence the map $\iota \kappa^{-1} : \kappa(\A) \to \ell^\infty(G,\A)$ extends to a unital completely positive $G$-equivariant map $\phi : I_G(\A) \to \ell^\infty(G,\A)$. This map is completely isometric on $\A$, and hence by the $G$-essentiality of $I_G(\A)$ it is completely isometric. Thus the image $\phi(I_G(\A))$ is a completely isometric $G$-equivariant copy of $I_G(\A)$ in $\ell^\infty(G,\A)$. We will see below that there are generally many $G$-equivariant copies of $I_G(\A)$ in $\ell^\infty(G,\A)$. 

If we identify $I_G(\A)$ with any completely isometric $G$-equivariant copy of itself in $\ell^\infty(G,\A)$, then by the $G$-injectivity of $I_G(\A)$, there is an idempotent unital completely positive $G$-equivariant map $\psi :\ell^\infty(G,\A) \to I_G(\A)$. Since $\ell^\infty(G,\A)$ is injective, it follows that $I_G(\A)$ is also injective. Hence by a result of Choi and Effros \cite{ChoEff}, the $G$-injective envelope $I_G(\A)$ is itself a C*-algebra with respect to the \emph{Choi-Effros product},
\[
A \cdot B = \psi(\psi(A)\psi(B)), \quad A,B \in I_G(\A).
\]
The C*-algebra obtained in this way is unique up to isomorphism, and in particular does not depend on the map $\psi$. We will identify $I_G(\A)$ with this abstract C*-algebra.

\section{The Furstenberg boundary} \label{sec:op-alg-construction}

\subsection{The Hamana boundary} \label{sec:defn-hamana-boundary}

Let $G$ be a discrete group. We want to consider the minimal C*-subalgebra of $\ell^\infty(G)$ that arises as the image of a unital positive $G$-equivariant projection on $\ell^\infty(G)$. By the discussion in Section \ref{sec:G-inj-envelope}, this is precisely the $G$-injective envelope $I_G(\bC)$ of $\bC$, where $\bC$ is equipped with the trivial $G$-action.

We know from Section \ref{sec:G-inj-envelope} that $I_G(\bC)$ is a unital C*-algebra, and it follows immediately from the commutativity of $\ell^\infty(G)$ and the definition of the product on $I_G(\bC)$ that it is also commutative. In particular, $I_G(\bC)$ can be identified with the algebra of continuous functions on a compact Hausdorff space.

\begin{defn}
Let $G$ be a discrete group. The \emph{Hamana boundary} $\hb$ of $G$ is the compact space such that $I_G(\bC) = C(\hb)$. By contravariance, the $G$-action on $C(\hb)$ induces a $G$-action on $\hb$ which we will refer to as the \emph{$G$-action on $\hb$}.
\end{defn}

We will soon prove the Hamana boundary $\hb$ can be identified with Furstenberg's universal $G$-boundary $\fb$. But first, we pause to observe that the size of the Hamana boundary can be viewed as a measure of the non-amenability of $G$. Recall that $G$ is \emph{amenable} if there is a positive unital $G$-equivariant map from $\ell^\infty(G)$ to $\bC$. If we identify the algebra $\bC$ with a subalgebra of $\ell^\infty(G)$, then this is equivalent to the existence of a positive unital $G$-equivariant projection from $\ell^\infty(G)$ onto $\bC$. Thus we obtain the following result (see also \cite{Ham78}*{Section 4}).

\begin{prop} \label{prop:trivial-iff-amenable}
Let $G$ be a discrete group. If the Hamana boundary $\hb$ is trivial, then $C(\hb) = \bC$ and $G$ is amenable. Otherwise, if $G$ is non-amenable, then $C(\hb) \ne \bC$, and $\hb$ is necessarily non-trivial.
\end{prop}

We will say more about the size of the Hamana boundary, and hence the size of the Furstenberg boundary, in Section \ref{sec:size-univ-boundary}.

\subsection{Boundaries} \label{sec:boundaries}

Let $G$ be a discrete group. In this section we will prove the Hamana boundary $\hb$ is a $G$-boundary in the sense of Furstenberg \cite{Furs73}.

\begin{defn} \label{defn:minimal-action}
Let $G$ be a discrete group, and let $X$ be a compact $G$-space. The $G$-action on $X$ is \emph{minimal} if for every $x$ in $X$, the $G$-orbit $Gx = \{sx \mid s \in G\}$ is dense in $X$.
\end{defn}

\begin{prop} \label{prop:action-is-minimal}
Let $G$ be a discrete group. Then the $G$-action on the Hamana boundary $\hb$ is minimal.
\end{prop}

\begin{proof}
For every $x \in \partial_H G$, the restriction map $C(\partial_H G) \to C(\overline{Gx})$ is unital and positive, where $\overline{Gx}$ denotes the closure of the orbit $Gx$. Hence by the $G$-essentiality of $C(\hb)$, it is  isometric, and in particular is injective. By contravariance, the inclusion map $\overline{Gx} \to \hb$ is surjective, and it follows that $Gx$ is dense in $\hb$.
\end{proof}

\begin{defn}
Let $G$ be a group, and let $X$ be a compact $G$-space. The $G$-action on $X$ is \emph{strongly proximal} if for every probability measure $\nu \in \P(X)$, the weak* closure of the $G$-orbit $G \nu$ contains a point mass $\delta_x \in \P(X)$ for some $x \in X$.
\end{defn}

\begin{lem} \label{lem:poisson-map-isometry}
Let $G$ be a discrete group, let $\hb$ denote the Hamana boundary of $G$, and let $\nu \in \P(\hb)$ be a probability measure. Then the Poisson map $P_\nu : C(\hb) \to \ell^\infty(G)$ defined in (\ref{eqn:poisson-map}) is an isometry.
\end{lem}

\begin{proof}
Since the map $P_\nu$ is unital, positive and $G$-equivariant, the result follows immediately from the $G$-essentiality of $C(\hb)$.
\end{proof}

The following result can be proved using Lemma \ref{lem:poisson-map-isometry} and a theorem of Azencott \cite{Aze70}. However, the proof of Azencott seems to use the commutativity of the C*-algebra $C(\hb)$ in an essential way, and therefore does not generalize to the non-commutative setting. We give a completely different proof which does generalize to the noncommutative setting. We believe this could be important for developing a notion of ``noncommutative Furstenberg boundary'' for a quantum group.

\begin{prop} \label{prop:hamana-strongly-proximal}
Let $G$ be a discrete group. Then the $G$-action on the Hamana boundary $\hb$ is strongly proximal.
\end{prop}

\begin{proof}
Let $\nu \in \P(\hb)$ be a probability measure. By Proposition \ref{prop:action-is-minimal}, the $G$-action on $\hb$ is minimal. Hence we must prove that for every point $x \in \hb$, the point mass $\delta_x \in \P(\hb)$ is contained in the weak* closure of the $G$-orbit $G\nu$.

Let $K$ denote the weak* closed convex hull of $G\nu$ and fix $x \in \hb$. We claim $\delta_x \in K$. Indeed, supposing otherwise, we can apply the Hahn-Banach separation theorem to find a positive function $f \in C(\hb)$ and $\epsilon > 0$ such that
\[
\langle f, t\nu \rangle \leq \langle f, \delta_x \rangle - \epsilon \leq \|f\| - \epsilon, \quad \forall t \in G.
\]
This implies that the Poisson map $P_\nu : C(\hb) \to \ell^\infty(G)$ satisfies
\[
P_\nu f (t) = \langle f, t\nu \rangle \leq \|f\| - \epsilon, \quad \forall t \in G. 
\]
But by Lemma \ref{lem:poisson-map-isometry}, $P_\nu$ is an isometry, which gives a contradiction. Thus $\delta_x$ is contained in $K$.

Since $\delta_x$ is an extreme point of the compact convex set $\P(\hb)$, and since $x$ was arbitrary, $K = \P(\hb)$. Hence by Milman's converse to the Krein-Milman theorem, it follows that the weak* closure of $G\nu$ contains the point mass $\delta_x$ for every $x \in \hb$.
\end{proof}

\begin{defn}
Let $G$ be a group. A compact $G$-space $X$ is said to be a \emph{$G$-boundary} if the $G$-action on $X$ is both minimal and strongly proximal.
\end{defn}

The following result follows immediately from Proposition \ref{prop:action-is-minimal} and Proposition \ref{prop:hamana-strongly-proximal}. 

\begin{cor}
Let $G$ be a discrete group. Then the Hamana boundary $\hb$ is a $G$-boundary.
\end{cor}

\subsection{Universality} \label{sec:universality}

Furstenberg proved in \cite{Furs73}*{Proposition 4.6}  that every group $G$ has a unique $G$-boundary $\fb$ which is \emph{universal}, in the sense that every $G$-boundary is a continuous $G$-equivariant image of $\fb$. We will refer to $\fb$ as the \emph{Furstenberg boundary of $G$}. In this section we will prove that the Hamana boundary $\hb$ can be identified with the Furstenberg boundary $\fb$.

\begin{lem} \label{lem:map-boundary-to-minimal-mult}
Let $G$ be a group, let $M$ be a minimal compact $G$-space and let $B$ be a compact $G$-boundary. There is at most one unital positive $G$-equivariant map from $C(B)$ to $C(M)$, and if such a map exists, then it is a unital injective *-homomorphism.
\end{lem}

\begin{proof}
Let $\phi : C(B) \to C(M)$ be a unital $G$-equivariant map. Consider the adjoint map $\phi^* : \M(M) \to \M(B)$, where $\M(M)$ and $\M(B)$ denote the spaces of regular Borel measures on $M$ and $B$ respectively. Restricting $\phi^*$ to the space of point masses on $M$ gives a continuous $G$-equivariant map $\alpha : M \to \P(B)$, where $\P(B)$ denotes the space of probability measures on $B$.

Since $M$ is minimal and $B$ is a boundary, it follows from \cite{Furs73}*{Proposition 4.2} that the range of $\alpha$ is the space of point masses on $B$, and that moreover that $\alpha$ must be unique. In particular, $\alpha$ can be identified with the unique continuous $G$-equivariant map from $M$ onto $B$

By contravariance, $\alpha$ induces an injective *-homomorphism from $C(B)$ to $C(M)$, and it is easy to check that this map necessarily agrees with $\phi$.
\end{proof}

We are grateful to Masamichi Hamana, who contacted us after receiving an earlier draft of this paper to point out the relevance of the reference \cite{Ham78}. In this work, which does not seem to be well known, Hamana constructs the injective envelope of a Banach module equipped with the action of a discrete group $G$. If the Banach module is a unital commutative C*-algebra, say $C(X)$ for some compact $G$-space $X$, then this injective envelope is precisely the $G$-injective envelope $I_G(C(X))$ of $C(X)$.

In particular, we discovered that the following result was stated without proof by Hamana in \cite{Ham78}*{Remark 4}. A proof can be deduced from Hamana's work using a dynamical characterization of the Furstenberg boundary contained in Glasner's book \cite{Gla76}, however the proof we give here has an operator-theoretic flavour.

\begin{thm} \label{thm:hamana-furstenberg-boundaries-agree}
Let $G$ be a discrete group, and let $B$ be a $G$-boundary, i.e. a minimal strongly proximal compact $G$-space. Then there is a continuous $G$-equivariant map from the Hamana boundary $\hb$ onto $B$. Hence the Hamana boundary $\hb$ can be identified with the Furstenberg boundary $\fb$.
\end{thm}

\begin{proof}
We must prove there is a continuous $G$-equivariant map from $\hb$ onto $B$. Make the identification $\ell^\infty(G) = C(\beta G)$, where $\beta G$ denotes the Stone-Cech compactification of $G$. Fix any point $x \in B$ and consider the continuous $G$-equivariant map $s \to sx,\ s \in G$. This extends to a continuous $G$-equivariant map  $\tau : \beta G \to B$. Moreover, by the minimality of the $G$-action on $B$ and the compactness of $\beta G$, the image of $\tau$ is both dense and compact. Hence $\tau$ maps $\beta G$ onto $B$.

By contravariance, $\tau$ induces a unital isometric $G$-equivariant *-homomorphism from $C(B)$ into $\ell^\infty(G)$. Identify $C(B)$ with its image under this map. 

By the $G$-injectivity of $C(\hb)$, there is a unital positive $G$-equivariant map from $\ell^\infty(G)$ onto $C(\hb)$. Let $\psi : C(B) \to C(\hb)$ denote the restriction to $C(B)$ of this map. Since $B$ is a boundary and $\hb$ is minimal, Lemma \ref{lem:map-boundary-to-minimal-mult} implies $\psi$ is an injective *-homomorphism. Applying contravariance again, $\psi$ induces a continuous $G$-equivariant map from $\hb$ onto $B$. 
\end{proof}

For the remainder of this paper, we will refer to the Furstenberg boundary $\fb$, instead of the Hamana boundary $\hb$. However, we will continue to make use of the operator-algebraic construction of $\hb$, as well as the results in Section \ref{sec:defn-hamana-boundary}.

\subsection{Rigidity}

Let $G$ be a discrete group, and let $\fb$ denote the Furstenberg boundary of $G$. The construction of $\fb$ using the theory of injective envelopes implies some powerful rigidity results for $G$-equivariant maps on $C(\fb)$, and by contravariance, for $G$-equivariant maps on $\fb$. We collect some of these results in the following theorem.

\begin{thm} \label{thm:rigidity}
Let $G$ be a discrete group, and let $\fb$ denote the Furstenberg boundary of $G$. The algebra $C(\fb)$ of continuous functions is $G$-injective, and in particular is injective. Moreover, we have the following rigidity results:
\begin{enumerate}
\item Every unital positive $G$-equivariant map from $C(\fb)$ into a $G$-operator system $S$ is completely isometric.
\item The only unital positive $G$-equivariant map from $C(\fb)$ to itself is the identity map.
\item If $B$ is a compact $G$-boundary, then there is a unique unital $G$-equivariant map from $C(B)$ to $C(\fb)$, and it is a unital injective *-homomorphism.
\item If $M$ is a minimal compact $G$-space, then there is at most one unital $G$-equivariant map from $C(\fb)$ to $C(M)$, and if such a map exists, then it is a unital injective *-homomorphism.
\end{enumerate}
\end{thm}

\begin{proof}
The first two results follows immediately from Theorem \ref{thm:hamana-furstenberg-boundaries-agree} and the $G$-essentiality and $G$-rigidity of $C(\fb)$. The second two results follow immediately from Lemma \ref{lem:map-boundary-to-minimal-mult} and the fact that $\fb$ is a $G$-boundary.
\end{proof}

\subsection{Copies of the Furstenberg boundary} \label{sec:copies-boundary}
Let $G$ be a discrete group, and let $\fb$ denote the Furstenberg boundary of $G$. We know from Theorem \ref{thm:hamana-furstenberg-boundaries-agree} and the results in Section \ref{sec:G-inj-envelope} that there is at least one unital positive $G$-equivariant map from $C(\fb)$ into $\ell^\infty(G)$, and the image of $C(\fb)$ under this map is an isometric $G$-equivariant copy of $C(\fb)$. However, in general there will be many isometric $G$-equivariant copies of $C(\fb)$ in $\ell^\infty(G)$. In this section we will give a complete description of these copies.

Let $\nu \in \P(\fb)$ be a probability measure, and consider the Poisson map $P_\nu : C(\fb) \to \ell^\infty(G)$ defined in (\ref{eqn:poisson-map}). By Lemma \ref{lem:poisson-map-isometry}, $P_\nu$ is an isometry, and hence the image $P_\nu(C(\fb))$ is an isometric $G$-equivariant copy of $C(\fb)$. The next result shows this correspondence is bijective.

\begin{prop}\label{G-equi-copies}
Let $G$ be a discrete group, and let $\fb$ denote the Furstenberg boundary of $G$. The map taking a probability measure $\nu \in \P(\fb)$ to $P_\nu(C(\fb))$ is a bijection between $\P(\fb)$ and the collection of isometric $G$-equivariant copies of $C(\fb)$ in $\ell^\infty(G)$. The image $P_\nu(C(\fb))$ is a C*-subalgebra if and only if $\nu$ is a point mass.
\end{prop}

\begin{proof}
Let $\mu, \nu \in \P(\fb)$ be probability measures. If $P_\mu(C(\fb)) = P_\nu(C(\fb))$, then $P_\nu^{-1}P_\mu(C(\fb)) = C(\fb)$, where $P_\nu^{-1}$ is restricted to the image of $P_\nu$. By Theorem \ref{thm:rigidity}, $P_\nu^{-1}P_\mu$ must be the identity map on $C(\fb)$, and hence $P_\mu = P_\nu$. Thus for $f \in C(\fb)$,
\[
\int_{\fb} f(x) \, d\mu(x) = P_\mu(f)(e) = P_\nu(f)(e) = \int_{\fb} f(x) \, d\nu(x),
\]
where $e$ denotes the identity element in $G$, and it follows that $\mu = \nu$. 

Now let $\psi : C(\fb) \to \ell^\infty(G)$ be an isometric $G$-equivariant map. Let $\delta_e \in \P(G)$ denote the point mass corresponding to the identity element in $G$, and define a probability measure $\nu$ on $\fb$ by $\nu = \psi^*(\delta_e)$. Then for $f \in C(\fb)$ and $t \in G$,
\[
P_\nu(f)(t) = \langle f, t\nu \rangle = \langle f, t\psi^*(\delta_e) \rangle = \langle \psi(f), \delta_t \rangle = \psi(f)(t).
\] 
Hence $\psi = P_\nu$. 

If $\nu$ is a point mass, then it is easy to check that $P_\nu$ is a *-homomorphism. Conversely, if $P_\nu$ is a *-homomorphism, then
\[
\langle fg, \nu \rangle = P_\nu(fg)(e) = P_\nu(f)(e)P_\nu(g)(e) = \langle f, \nu \rangle \langle g, \nu \rangle.
\]
Therefore, $\nu$ corresponds to a multiplicative state on $C(\fb)$, and hence $\nu$ must be a point mass.
\end{proof}

The next result is essentially \cite{HadPau11}*{Proposition 4.1}.

\begin{prop} \label{prop:zero-intersection-c-zero}
Let $G$ be a non-amenable discrete group, and let $\fb$ denote the Furstenberg boundary of $G$. For every isometric $G$-equivariant embedding $C(\fb) \subset \ell^\infty(G)$, we have $C(\fb) \cap c_0(G) = \{0\}$.
\end{prop}

\begin{proof}
The space $c_0(G)$ is $G$-invariant, and $G$ is infinite since it is non-amenable. Hence $\bC \not \subset c_0(G)$. The quotient map $\ell^\infty(G) \to \ell^\infty(G)/c_0(G)$ to $\bC$ is a unital positive $G$-equivariant map which is isometric on $\bC$. By Theorem \ref{thm:rigidity}, it follows that it is isometric on $C(\fb)$.
\end{proof}

\subsection{The size of the Furstenberg boundary} \label{sec:size-univ-boundary}

In this section we consider some facts about the size of the Furstenberg boundary.

We are grateful to R\'emi Boutonnet for pointing out a flaw in an earlier proof of the following result, and suggesting an appropriate correction.

\begin{prop} \label{prop:no-isolated-points}
Let $G$ be a discrete group, and let $\fb$ denote the Furstenberg boundary of $G$. If $G$ is non-amenable, then $\fb$ does not contain any isolated points.
\end{prop}

\begin{proof}
Suppose $G$ is non-amenable and suppose for the sake of contradiction that $x \in \fb$ is an isolated point. By Proposition \ref{prop:action-is-minimal}, the $G$-action on $\fb$ is minimal. Hence for every $y \in \fb$, the $G$-orbit $Gy$ is dense in $\fb$, which implies $x \in Gy$. In particular, there is $s \in G$ such that $sy = x$. Thus $y$ is also isolated, and since $y$ was arbitrary, it follows that every point of $\fb$ is isolated.

Since $\fb$ is compact, it must be finite, and it is easy to construct a unital positive $G$-invariant map $\phi : C(\fb)$ onto $\bC$. By Theorem \ref{thm:rigidity}, there is a unital positive $G$-equivariant map  $\psi : \ell^\infty(G) \to C(\fb)$. The composition $\phi \psi : \ell^\infty(G) \to \bC$ gives a unital positive $G$-invariant map, contradicting the fact that $G$ is non-amenable.
\end{proof}

\begin{rem} \label{rem:univ-bound-is-stonean}
It follows from Theorem \ref{thm:rigidity} that the algebra $C(\fb)$ is injective, and hence is an AW* algebra. In particular, $\fb$ is a Stonean space (see e.g. \cite{Ham79a}*{Proposition 4.15}). This implies that if $\fb$ is infinite, then it is not second countable, and hence is non-metrizable. In this case, $C(\fb)$ is non-separable.
\end{rem}

The next result follows immediately from Proposition \ref{prop:trivial-iff-amenable}, Proposition \ref{prop:no-isolated-points} and Remark \ref{rem:univ-bound-is-stonean}.

\begin{cor} \label{cor:infinite-non-separable}
Let $G$ be a discrete group. Then the Furstenberg boundary $\fb$ is either trivial, i.e. a singleton, or non-metrizable. Hence $C(\fb) = \bC$ if $G$ is amenable, and $C(\fb)$ is non-separable if $G$ is non-amenable.
\end{cor}

\begin{rem}
Furstenberg proved in \cite{Furs63}*{Theorem 1.5} that if $G$ is a semisimple Lie group, then the Furstenberg boundary $\fb$ is a homogeneous space, and in particular is metrizable. However, if $G$ is non-amenable, then attempting to duplicate the operator-algebraic construction of $\fb$ in Section \ref{sec:op-alg-construction} results in a non-separable space as in Corollary \ref{cor:infinite-non-separable}. Thus the Hamana boundary and the Furstenberg boundary do not coincide in this case.
\end{rem}

Day showed in \cite{Day57}*{Lemma 4.1} that every discrete group $G$ has a largest amenable normal subgroup $R_a(G)$ called the {\it amenable radical of $G$} that contains every amenable normal subgroup of $G$.

\begin{prop}\label{hb-same-for-amenable-radical}
Let $G$ be a discrete group, let $R_a(G)$ denote the amenable radical of $G$, and let $G' = G/R_a(G)$. Then $\fb = \partial_F G'$.
\end{prop}

\begin{proof}
We begin by noting that there are natural actions of $G$ on $\partial_F G'$ and $G'$ on $\fb$.

Fix a point $y \in \fb$, and let $\delta_y$ denote the corresponding point mass. Consider the Poisson map $P_{\delta_y} : C(\fb) \to \ell^\infty(G)$ defined as in (\ref{eqn:poisson-map}). For $f \in C(\fb)$ and $t \in G$, we have
\[
P_{\delta_y}(f)(t) = \int_{\fb} f(tx) \, \delta_y(dx) = f(ty).
\]
Since the amenable radical $R_a(G)$ acts trivially on $\fb$ (\cite{Furm}*{Corollary 8}), for $t \in R_a(G)$ we have $P_{\delta_y}(f)(t) = f(y)$. In particular, $P_{\delta_y}(f)$ is constant on the cosets of $R_a(G)$. It follows there is a natural map from $P_{\delta_y}(C(\fb))$ into $\ell^\infty(G')$. Moreover, this map is $G$-equivariant. Composing with the map from $\ell^\infty(G')$ onto $C(\partial_F G')$ gives a $G$-equivariant map $\phi : C(\fb) \to C(\partial_F G')$.

On the other hand, we can compose the inclusion map of $C(\partial_F G')$ into $\ell^\infty(G')$, which is $G$-equivariant, with the natural map into $\ell^\infty(G)$, and then with the $G$-idempotent map from $\ell^\infty(G)$ onto $C(\fb)$ to get a $G$-equivariant map $\psi : C(\partial_F G') \to C(\fb)$.

By Theorem \ref{thm:rigidity}, the maps $\phi$ and $\psi$ are isometries, and the maps $\psi \phi$ and $\phi \psi$ must be the identity maps on $C(\fb)$ and $C(\partial_F G')$ respectively. It follows by contravariance that $\fb$ and $\partial_F G'$ are homeomorphic as $G$-spaces.
\end{proof}

\section{Exactness and nuclear embeddings} \label{sec:exactness-and-nuclear-embeddings}

\subsection{Exactness and amenable actions} \label{sec:exactness-and-amenable-actions}

Let $G$ be a discrete group, and let $\fb$ denote the Furstenberg boundary of $G$. In this section we will show that if $G$ is exact, then the $G$-action on $\fb$ is amenable. Since $G$ is exact if it acts amenably on some compact $G$-space, it follows that the amenability of the $G$-action on $\fb$ completely characterizes the exactness of $G$. For the general theory of exactness and amenable actions, we refer the reader to the book of Brown and Ozawa \cite{BroOza08}.

The definition of an amenable group action is due to Anantharaman-Delaroche \cite{Del02}*{Definition 2.1}.

\begin{defn}\label{defn:amenable-action}
Let $G$ be a discrete group, and let $X$ be a compact $G$-space. The $G$-action on $X$ is \emph{amenable} if there is a net of continuous maps $m_i : X \to \P(G)$ such that
\begin{equation}
\lim_i \sup_{x \in X} \| s m_i(x) - m_i(s x) \|_1 = 0, \quad \forall s \in G, \label{eqn:amenable-action}
\end{equation}
where the space $\P(G)$ of probability measures on $G$ is endowed with the weak* topology.
\end{defn}

The notion of an exact group is due to Kirchberg and Wasserman \cite{KW99}.
\begin{defn}\cite{KW99}
A discrete group G is \emph{exact} if the reduced C*-algebra $\ca_r(G)$ of $G$ is exact.
\end{defn}

Although exactness is an operator-algebraic property, it is intimately connected with the ergodic and geometric properties of the group. This can be seen from the following characterization (see \cite{Oza00} or \cite{BroOza08}*{Theorem 5.1.6}).

\begin{thm} \label{BOthm5.1.7}
Let $G$ be a discrete group. Then $G$ is exact if and only if it acts amenably on some compact space.
\end{thm}

We are grateful to Nicolas Monod for pointing out that the next Lemma can be seen as a special case of \cite{CM14}*{Proposition 9}, and to Narutaka Ozawa for pointing out that it also appears in \cite{H2000}*{Lemma 3.6}.

\begin{lem}\label{state-action}
Let $G$ be a discrete group, and let $X$ be a compact $G$-space. If the $G$-action on $X$ is amenable, then so is the corresponding $G$-action on the space $\P(X)$ of probability measures on $X$, endowed with the weak* topology.
\end{lem}

\begin{thm}\label{cor:action-amenable}
Let $G$ be a discrete group, and let $\fb$ denote the Furstenberg boundary of $G$. Then $G$ is exact if and only if the $G$-action on $\fb$ is amenable.
\end{thm}

\begin{proof}
If the $G$-action on $\fb$ is amenable, then $G$ is exact by Theorem \ref{BOthm5.1.7}. We must prove the converse, i.e. that if $G$ is exact, then the $G$-action on $\fb$ is amenable.

By the $G$-injectivity of $C(\fb)$ there is a unital positive $G$-equivariant map $\psi : \ell^\infty(G) \to C(\fb)$. If we identify $\ell^\infty(G)$ with $C(\beta G)$, where $\beta G$ denotes the Stone-Cech compactification of $G$, then the adjoint map $\psi^* : \M(\fb) \to \M(\beta G)$ is $G$-equivariant, and the restriction of $\psi^*$ to the space of point masses on $\fb$ gives a continuous $G$-equivariant map $\alpha : \fb \to \P(\beta G)$, where $\P(\beta G)$ denotes the space of probability measures on $\beta G$, endowed with the weak* topology.

Since $G$ is exact, the $G$-action on $\beta G$ is amenable by a result of Ozawa \cite{Oza00}. Hence by Lemma \ref{state-action}, the $G$-action on $\P(\beta G)$ is amenable. It follows from the existence of the $G$-equivariant map $\alpha$ that the $G$-action on $\fb$ is also amenable.
\end{proof}

\subsection{A canonical nuclear embedding} \label{sec:canonical-nuc-embedding}

In this section we prove that the reduced C*-algebra of every discrete exact group has a canonical embedding into a nuclear C*-algebra which is a subalgebra of the injective envelope.

\begin{thm} \label{thm:nuclear-embedding}
Let $G$ be a discrete exact group, and let $\ca_r(G)$ denote the corresponding reduced C*-algebra. There is a canonical nuclear C*-algebra $N(\ca_r(G))$ such that 
\[
\ca_r(G) \subset N(\ca_r(G)) \subset I(\ca_r(G)),
\]
where $I(\ca_r(G))$ denotes the injective envelope of $\ca_r(G)$. The algebra $N(\ca_r(G))$ is simple if $\ca_r(G)$ is simple, and prime if and only if $\ca_r(G)$ is prime.
\end{thm}

\begin{proof}
Since $G$ is exact, Corollary \ref{cor:action-amenable} implies that the $G$-action on the Furstenberg boundary $\fb$ is amenable. Thus by \cite{BroOza08}*{Theorem 4.3.4}, the reduced crossed product $\rcp{C(\fb)}{G}$ is nuclear. Identifying $\ca_r(G)$ with $\rcp{\bC}{G}$ and applying \cite{Ham85}*{Theorem 3.4} implies
\[
\ca_r(G) = \rcp{\bC}{G} \subset \rcp{C(\fb)}{G} \subset I(\rcp{\bC}{G}) = I(\ca_r(G)).
\]
Hence we take $N(\ca_r(G)) = \rcp{C(\fb)}{G}$. The statement about the simplicity and primeness of $N(\ca_r(G))$ follows from a result of Hamana \cite{Ham85}*{Corollary 3.5}. 
\end{proof}

\begin{rem}
We note that if $G$ has the infinite conjugacy class property, i.e. is an icc group, then $\ca_r(G)$ is prime (cf. \cite{Mur03}), and hence the algebra $N(\ca_r(G))$ in Theorem \ref{thm:nuclear-embedding} is prime.
\end{rem}

\begin{rem}
It follows from Corollary \ref{cor:infinite-non-separable} that the algebra $N(\ca_r(G))$ in Theorem \ref{thm:nuclear-embedding} is not separable if $G$ is non-amenable. We view this as a consequence of the fact that this construction works for non-hyperbolic groups. However, it is known that a separable C*-subalgebra of a nuclear C*-algebra is always contained in a separable nuclear C*-subalgebra (see e.g. \cite{BroOza08}*{Example 2.3.8}). Thus we can replace $N(\ca_r(G))$ with a separable C*-algebra if we do not require the algebra to be canonical.
\end{rem}

\section{Rigidity and the injective envelope} \label{sec:rigidity-stationary-measures}

A key step in Ozawa's paper is a rigidity result \cite{Oza07}*{Proposition 3} for positive equivariant maps between spaces of functions on the hyperbolic boundary of the free group. Ozawa observes that his result extends to hyperbolic groups. In this section we prove a generalization of this result which imposes slightly weaker requirements. The main advantage of our approach is that it extends to groups that are not necessarily hyperbolic or even relatively hyperbolic, including certain mapping class groups.

\begin{defn} \label{defn:stationary-measure}
Let $G$ be a discrete group, let $X$ be a compact $G$-space, and let $\mu \in \P(G)$ be a probability measure. A probability measure $\nu \in \P(X)$ is \emph{$\mu$-stationary} if $\mu \ast \nu = \nu$, where
\[
\mu \ast \nu = \sum_{s \in G} \mu(s) \, s \nu.
\]
\end{defn}

\begin{defn} \label{defn:harmonic-functions}
Let $G$ be a discrete group, and let $\mu \in \P(G)$ be a probability measure. A function $f \in \ell^\infty(G)$ is said to be $\mu$-harmonic if
\[
f(s) = \sum_{t \in G} f(st) \, \mu(t), \quad \forall s \in G.
\]
\end{defn}

For a probability measure $\mu \in \P(G)$, the space $H^\infty(G,\mu)$ of $\mu$-harmonic functions in $\ell^\infty(G)$ is 
a weak*-closed operator subsystem of $\ell^\infty(G)$, and there is a unital positive idempotent $G$-equivariant map from $\ell^\infty(G)$ onto $H^\infty(G,\mu)$. Equipped with the corresponding Choi-Effros product, $H^\infty(G,\mu)$ is a commutative von Neumann algebra (although not, in general, a subalgebra of $\ell^\infty(G)$). In particular, there is a compact $G$-space $\Pi_\mu$, called the \emph{topological Poisson boundary} of the pair $(G, \mu)$, such that $H^\infty(G,\mu) = C(\Pi_\mu)$.

Let $e$ denote the identity element of $G$, and let $\delta_e \in \P(G)$ denote the corresponding point mass. Then  $H^\infty(G,\mu) \cong L^\infty(\Pi_\mu, \mu_\infty)$, where $\mu_\infty$ is the probability measure on $\Pi_\mu$ obtained by restricting the state $\delta_e$ on $\ell^\infty(G)$ to $H^\infty(G,\mu)$. The probability space $(\Pi_\mu, \mu_\infty)$ is called the \emph{Poisson boundary} of the pair $(G, \mu)$.

\begin{thm}\label{prop3}
Let $G$ be a non-amenable hyperbolic group, and let $\mu \in \P(G)$ be an irreducible probability measure. Let $\partial G$ denote the hyperbolic boundary of $G$, and let $\nu \in \P(\partial G)$ be a $\mu$-stationary probability measure. If $\phi : C(\partial G) \to L^\infty(\partial G, \nu)$ is a unital positive $G$-equivariant map, then $\phi = \operatorname{id}$.
\end{thm}

\begin{proof}
By the results of Kaimanovich in \cite{Kai00}, the $\mu$-stationary probability measure $\nu$ is unique, and $(\partial G, \nu)$
is a $\mu$-boundary in the sense of Furstenberg \cite{Furs73}. By \cite{Furs73}*{Theorem 12.2}, $(\partial G, \nu)$ is a quotient of the Poisson boundary $(\Pi_\mu,\mu_\infty)$. Therefore, there is an embedding of $L^\infty(\partial G, \nu)$ into $L^\infty(\Pi_\mu,\mu_\infty)$.

A straightforward computation shows that the Poisson map $P_\nu : L^\infty(\partial G, \nu) \to \ell^\infty(G)$ defined as in (\ref{eqn:poisson-map}) is, in fact, the composition of the above embedding with the Poisson map $P_{\mu_\infty} : L^\infty(\Pi_\mu,\mu_\infty) \to H^\infty(G,\mu) \subset \ell^\infty(G)$. In particular, $P_\nu$ is isometric. Therefore, by a result of Jaworski \cite{Jaworski}*{Corollary 2.4}, the measure $\nu$ is strongly approximately transitive, which means that the convex hull of the $G$-orbit $G\nu$ is dense in $\P(\partial G, \nu)$ with respect to the total variation norm, where $\P(\partial G, \nu)$ denotes the space of probability measures on $\partial G$ that are absolutely continuous with respect to $\nu$.

Now, it is easy to check that the adjoint map $\phi^* : L^\infty(\partial G, \nu)^* \to \M(\partial G)$ maps $\nu$ to a $\mu$-stationary measure. Since $\nu$ is the unique $\mu$-stationary measure on $\partial G$, it follows that $\phi^*(\nu) = \nu$. Thus for $f \in C(\partial G)$,
\[
\langle\, s \nu, \phi(f)\, \rangle \,=\, \langle \,s \phi^*(\nu), f\, \rangle\, =\, \langle\, s \nu, f\, \rangle, \quad \forall s \in G.
\]
Since $\nu$ is strongly approximately transitive, the $G$-orbit $G\nu$ spans a norm dense subspace in $L^1(\partial G, \nu)$. Hence $\phi(f) = f$.
\end{proof}

\begin{rem}
We note that if the probability measure $\mu$ in the statement of \ref{prop3} is, in addition, symmetric, then by Kaimanovich \cite{Kai03}*{Theorem 3}, the stationary measure $\nu$ is doubly ergodic and Ozawa's argument from \cite{Oza07}*{Proposition 3} applies.
\end{rem}

We can now identify the $G$-injective envelope of the algebra $C(\partial G)$.

\begin{cor} \label{cor:identification-hamana-boundary}
Let $G$ be a hyperbolic group, and let $\partial G$ denote the hyperbolic boundary of $G$. Then $I_G(C(\partial G)) = C(\fb)$, where $I_G(C(\partial G))$ denotes the $G$-injective envelope of $C(\partial G)$, and $\fb$ denotes the Furstenberg boundary of $G$.
\end{cor}

\begin{proof}
The von Neumann algebra crossed product $L^\infty(\partial G, \nu) \rtimes G$ is injective, since the action of a hyperbolic group on its hyperbolic boundary is amenable by a result of Adams \cite{Ada94}. Therefore, applying Theorem \ref{prop3}, we can argue as in \cite{Oza07} that
\[
\rcp{C(\partial G)}{G} \subset I(\ca_r(G)) = I(\rcp{\bC}{G}).
\]
It follows from \cite{Ham85}*{Theorem 3.4} there is unital positive $G$-equivariant inclusion $\iota_0: C(\partial G) \hookrightarrow C(\fb)$, and by the $G$-injectivity of $C(\fb)$, we can extend $\iota_0$ to a unital positive $G$-equivariant map $\iota : I_G(C(\partial G)) \to C(\fb)$. 

On the other hand, since $\bC \subset I_G(C(\partial G))$, and since the latter space is $G$-injective, there is an isometric $G$-equivariant isometric inclusion $\eta : C(\fb) \hookrightarrow I_G(C(\partial G))$.

By the $G$-rigidity of the spaces $C(\fb)$ and $I_G(C(\partial G))$, the maps $\iota \eta$ and $\eta \iota$ must be the identity maps.
\end{proof}

\begin{rem} \label{rem:hyperbolic-boundary-is-g-boundary}
Corollary \ref{cor:identification-hamana-boundary} implies that there is a C*-algebra inclusion $C(\partial G) \subset C(\fb)$, and hence that the hyperbolic boundary $\partial G$ is a quotient of the Furstenberg boundary $\fb$. In particular, this implies that $\partial G$ is a $G$-boundary, i.e. the $G$-action on $\partial G$ is minimal and strongly proximal. This fact seems to be well known.
\end{rem}

In fact, our proof of Theorem \ref{prop3} implies that when $X$ is a compact $G$-space with a unique $\mu$-stationary probability measure $\nu \in \P(X)$ such that $\nu$ has full support and $(X, \nu)$ is a $\mu$-boundary, then the only unital positive $G$-equivariant map from $C(X)$ to $L^\infty(X,\nu)$ is the identity map. Note that Corollary \ref{cor:identification-hamana-boundary} also holds in this setting, provided that the action of $G$ on $(X,\nu)$ is amenable in the sense of Zimmer \cite{Zim78}, which is equivalent to the von Neumann algebra crossed product $L^\infty(X,\nu) \rtimes G$ being injective.

In particular, using the work of Kaimanovich and Masur \cite{Kai-Mas96} on mapping class groups and arguing as above, we can prove the following result.

\begin{thm} \label{thm:thurston-boundary}
Let $S$ be a closed orientable surface of genus $g\geq 2$. Let $G$ be the mapping class group of $S$, and let $\mu \in \P(G)$ be an irreducible probability measure. Let $\partial_T G$ denote the Thurston boundary of $G$, and let $\nu\in \P(\partial_T G)$ be a $\mu$-stationary probability measure on $\partial_T G$. Then we have the following results:
\begin{enumerate}
\item If $\phi : C(\partial_T G) \to L^\infty(\partial_T G, \nu)$ is a unital positive $G$-equivariant map, then  $\phi = \operatorname{id}$.
\item The $G$-injective envelope of $C(\partial_T G)$ is $C(\fb)$.
\end{enumerate}
\end{thm}

\begin{proof}
Kaimanovich and Masur proved in \cite{Kai-Mas96}*{Theorem 2.2.4} that there is a unique $\mu$-stationary measure $\nu$ on $\partial_T G$ with full support, and that $(\partial_T G, \nu)$ is a $\mu$-boundary. Hence an argument similar to the proof of Theorem \ref{prop3}  yields (1). 

It was further proved in \cite{Kai-Mas96}*{Theorem 2.2.4} that $\nu$ is concentrated on the space of minimal measured foliations, and hence by \cite{Kid08}*{Theorem 1.4}, the $G$-action on $(\partial_T G, \nu)$ is amenable in the sense of Zimmer \cite{Zim78}, which implies that the von Neumann algebra crossed product $L^\infty(\partial_T G, \nu) \rtimes G$ is injective.

The proof of (2) is similar to the proof of Corollary \ref{cor:identification-hamana-boundary}.
\end{proof}

\begin{rem}
We note that Theorem \ref{thm:thurston-boundary} similarly implies the well-known fact that the Thurston boundary is a $G$-boundary.
\end{rem}

\section{C*-simplicity}\label{sec:C*-sim}
A discrete group $G$ is \emph{C*-simple} if the reduced C*-algebra $\ca_r(G)$ of $G$ is simple, i.e. it has no non-trivial closed two-sided ideals. An equivalent reformulation in the language of representation theory is that every unitary representation of $G$ that is weakly contained in the left regular representation is actually weakly equivalent to the left regular representation. The problem of determining which groups are C*-simple has received a great deal of attention (see \cite{Del07} for a survey).

In this section, we will prove that $G$ is C*-simple if and only if the $G$-action on the Furstenberg boundary $\fb$ is topologically free. Thus the study of the $G$-action on $\fb$ provides a new approach to problems about C*-simplicity. As an application of these ideas, we will prove the C*-simplicity of Tarski monster groups.

\begin{defn} \label{defn:topological-freeness}
Let $G$ be a discrete group with identity element $e$, and let $X$ be a compact $G$-space. The $G$-action on $X$ is \emph{topologically free} if for every $s \in G \setminus \{e\}$, the set
\[
X \setminus X^s = \{x \in X \mid sx \ne x\}
\]
is dense in $X$.
\end{defn}

The following theorem is the main result in this section. We are grateful to Narutaka Ozawa for showing us a proof of the implication (2) $\Rightarrow$ (4) which does not require $G$ to be exact.

\begin{thm} \label{thm:c-star-simplicity}
Let $G$ be a discrete group, and let $\fb$ denote the Furstenberg boundary of $G$. Then the following are equivalent:
\begin{enumerate}
\item The group $G$ is C*-simple, i.e. the reduced C*-algebra $\ca_r(G)$ is simple.
\item The reduced crossed product $\rcp{C(\fb)}{G}$ is simple.
\item The reduced crossed product $\rcp{C(B)}{G}$ is simple for some $G$-boundary $B$.
\item The $G$-action on $\fb$ is topologically free.
\item The $G$-action on some $G$-boundary is topologically free.
\end{enumerate}
\end{thm}

\begin{proof}
(1) $\Rightarrow$ (2) If $\ca_r(G) = \rcp{\bC}{G}$ is simple, then $\rcp{C(\fb)}{G}$ is also simple by \cite{Ham85}*{Corollary 3.5}.


(2) $\Rightarrow$ (1) Suppose that $\rcp{C(\fb)}{G}$ is simple. Let $\J$ be an ideal of $\ca_r(G)$ and let $\pi : \ca_r(G) \to \ca_r(G)/{\J}$ denote the corresponding quotient map. Note that $G$ acts on $\ca_r(G)$ and $\ca_r(G)/{\J}$ by conjugation, i.e. $s \to \operatorname{Ad}(\lambda_s)$ and $s \to \operatorname{Ad}(\pi(\lambda_s))$ respectively for $s \in G$, where  $\lambda$ denotes the left regular representation of $G$.

Observe that $\pi$ is $G$-equivariant with respect to these actions. By $G$-injectivity, we can extend $\pi$ to a $G$-equivariant map $\tilde{\pi} : \rcp{C(\fb)}{G} \to I_G(\ca_r(G)/{\J})$. By rigidity, the $G$-action on  $I_G(\ca_r(G)/{\J})$ is also implemented by conjugation, i.e. $s \to \operatorname{Ad}(\pi(\lambda_s))$ for $s \in G$.

We would be done if $\tilde{\pi}$ was a *-homomorphism. Unfortunately, while the image of $C(\fb)$ under $\tilde{\pi}$ is a $G$-equivariant embedding of $C(\fb)$ into $I_G(\ca_r(G)/{\J})$, it may not be a *-algebra embedding. We will find a *-algebra embedding of $C(\fb)$ into the second dual $I_G(\ca_r(G)/\J)^{**}$, and use this embedding to construct a suitable *-homomorphism of $\rcp{C(\fb)}{G}$.

Applying $G$-injectivity again, there is a $G$-equivariant projection $\phi : I_G(\ca_r(G)/{\J}) \to \tilde{\pi}(C(\fb))$. Endowed with the corresponding Choi-Effros product, $\tilde{\pi}(C(\fb))$ is a C*-algebra isomorphic to $C(\fb)$.

The restriction of $\phi$ to the C*-algebra $\ca(\tilde{\pi}(C(\fb)))$ is easily seen to be a $G$-equivariant surjective *-homomorphism with respect to the Choi-Effros product on $\tilde{\pi}(C(\fb))$. Let $\rho : \ca(\tilde{\pi}(C(\fb))) \to C(\fb)$ denote the composition of this restriction with the *-isomorphism onto $C(\fb)$.

Passing to the bidual, the $G$-action on $I_G(\ca_r(G)/\J)$ canonically extends to a $G$-action on $I_G(\ca_r(G)/\J)^{**}$, i.e.  $s \to \operatorname{Ad}(\pi(\lambda_s))$ for $s \in G$. We have the inclusion $\ca(\tilde{\pi}(C(\fb)))^{**} \subset I_G(\ca_r(G)/\J)^{**}$, and $\ca(\tilde{\pi}(C(\fb)))^{**}$ is $G$-invariant.

Consider the normal $G$-equivariant surjective *-homomorphism
\[
\rho^{**} : \ca(\tilde{\pi}(C(\fb)))^{**} \to C(\fb)^{**}.
\]
Let $q \in \ca(\tilde{\pi}(C(\fb)))^{**} \subset I_G(\ca_r(G)/\J)^{**}$ denote the corresponding support projection. Then $\ca(\tilde{\pi}(C(\fb)))^{**}q \simeq C(\fb)^{**}$.

Define $\sigma : \rcp{C(\fb)}{G} \to I_G(\ca_r(G)/\J)^{**}$ by
\[
\sigma(x) = \tilde{\pi}(x)q, \quad x \in \rcp{C(\fb)}{G}.
\]
We claim that $\sigma$ is a *-homomorphism. To see this, first observe that since the kernel of $\rho^{**}$ is $G$-invariant, $q$ commutes with $\pi(\ca_r(G))$. Hence the restriction of $\sigma$ to $\ca_r(G)$ is a *-homomorphism. Next, observe that the restriction of $\sigma$ to $C(\fb)$ implements the *-isomorphism between $\ca(\tilde{\pi}(C(\fb)))^{**}q$ and $C(\fb)$. Since $\rcp{C(\fb)}{G}$ is generated by $C(\fb)$ and $\rcp{C(\fb)}{G}$, it follows that $\sigma$ is a *-homomorphism. 

Since $\rcp{C(\fb)}{G}$ is simple, $\ker \sigma = \{0\}$ is trivial. Hence $\J = \ker \pi = \{0\}$.


(2) $\Rightarrow$ (4) Suppose that the $G$-action on $\fb$ is not topologically free. We will prove that $\rcp{C(\fb)}{G}$ is not simple. The proof is modeled on the proofs of Archbold-Spielberg \cite{ArcSpi94}*{Theorem 2} and Kawamura-Tomiyama \cite{KT1990}*{Theorem 4.1}.

Fix $s \in G \setminus {e}$ such that $(\fb)^s = \{ x \in \fb \mid sx = x\}$ has non-empty interior, and $x \in (\fb)^s$. We first claim that the stabilizer subgroup $G_x = \{t \in G \mid tx = x \}$ is amenable.

Let $\phi : \ell^\infty(G) \to C(\fb)$ be a $G$-equivariant projection onto $C(\fb)$, and let $\delta_x : C(\fb) \to \bC$ denote the point mass on $C(\fb)$ corresponding to $x$. Then $\delta_x \circ \phi$ is a $G_x$ invariant state on $\ell^\infty(G)$.

Let $(s_\alpha)_{\alpha \in A}$ be a system of representatives for the right cosets $G_x \backslash G$. Define a map $\rho : \ell^\infty(G_x) \to \ell^\infty(G)$ by $\rho(f)(t) = f(r)$, where $t = rs_\alpha$ for $r \in G_x$. Then $\rho$ is a unital $G_x$-equivariant injective *-homomorphism, and the composition $\delta_x \circ \phi \circ \rho$ is a $G_x$-invariant projection on $\ell^\infty(G_x)$, which establishes the claim.

Since $G_x$ is amenable, the induced representation $\lambda'$ on $\ell^2(G/G_x)$ corresponding to the subgroup $G_x$ is weakly contained in the left regular representation of $G$. Define $\pi_x : C(\fb) \to \ell^\infty(G/G_x)$ by $\pi_x(f)(tG_x) = f(tx)$.  Then $\pi_x$ is a *-homomorphism and $(\lambda', \pi_x, \ell^2(G/G_x))$ is a covariant representation. We claim it gives a continuous representation of $\rcp{C(\fb)}{G}$.

By Fell's absorption principle (cf. \cite{BroOza08}*{Proposition 4.1.7}),
\[
\ca((\lambda' \otimes \lambda)(G),\pi_x(C(\fb)) \otimes 1_{\ell^2(G)}) \simeq \rcp{C(\fb)}{G}.
\]
Hence by the weak containment of $\lambda'$ in $\lambda$, the map sending $\lambda_t \to \lambda'_t \otimes \lambda'_t$ for $t \in G$ and $f \to \pi_x(f) \otimes 1_{\ell^2(G/G_x)}$ for $f \in C(\fb)$ is continuous on $\rcp{C(\fb)}{G}$. The diagonal subspace $\ell^2(G/G_x) \otimes \ell^2(G/G_x)$ reduces the image of this map. Compressing to this subspace and identifying it with $\ell^2(G/G_x)$ establishes the claim.

Now let $f \in C(\fb)$ be a function supported in $(\fb)^s$. Letting $\lambda' \times \pi_x$ denote the representation of $\rcp{C(\fb)}{G}$ induced by $(\lambda', \pi_x, \ell^2(G/G_x))$, it is easy to check that $(\lambda' \times \pi_x)(f(1-\lambda_s)) = 0$. Hence $\rcp{C(\fb)}{G}$ is not simple.


(4) $\Rightarrow$ (2) If the $G$-action on $\fb$ is topologically free, then it follows from a result of Archbold and Spielberg \cite{ArcSpi94}*{Corollary 1}, that $\rcp{C(\fb)}{G}$ is simple.


(5) $\Rightarrow$ (3) This implication also follows from \cite{ArcSpi94}*{Corollary 1}.


(3) $\Rightarrow$ (2) Let $X$ be a $G$-boundary with the property that $\rcp{C(X)}{G}$ is simple. By universality, $X$ is a continuous $G$-equivariant image of $\fb$. Hence there is a $G$-equivariant embedding of $C(X)$ as a subalgebra of $C(\fb)$. Since $\rcp{C(X)}{G} \subset \rcp{C(\fb)}{G}$, it follows from \cite{Ham85}*{Corollary 3.5} that $ \rcp{C(\fb)}{G}$ is simple. 


(4) $\Rightarrow$ (5) This implication is trivial.

\end{proof}

\begin{rem}
The following fact, extracted from the proof of the implication (2) $\Rightarrow$ (1) in Theorem \ref{thm:c-star-simplicity}, seems to be of independent interest: For a $G$-C*-algebra $\A$, if there is a $G$-equivariant embedding (i.e. a $G$-equivariant completely isometric map) of $\A$ into a $G$-C*-algebra $\B$, then there is a a $G$-equivariant *-algebraic embedding (i.e. a $G$-equivariant injective *-homomorphism) of $\A$ into the bidual $\B^{**}$.
\end{rem}

We recall from Section \ref{sec:size-univ-boundary} that the amenable radical $R_a(G)$ of $G$ is the largest normal amenable subgroup of $G$ \cite{Day57}*{Lemma 4.1}. It is known that if $G$ is C*-simple, then $R_a(G)$ is trivial. It is a longstanding open problem to determine whether the converse of this result is true, i.e. whether the triviality of $R_a(G)$ implies the C*-simplicity of $G$.

The following result follows immediately from Proposition \ref{hb-same-for-amenable-radical} (see also \cite{Furm}*{Corollary 8}).

\begin{prop}\label{top-free-implies-triv-amen-rad}
Let $G$ be a discrete group, and let $\fb$ denote the Furstenberg boundary of $G$. If the $G$-action on $\fb$ is topologically free, then the amenable radical $R_a(G)$ of $G$ is trivial.
\end{prop}

In \cite{Del07}, de la Harpe asked if there exists a a countable group which is both C*-simple and does not contain non-abelian free subgroups. This question was answered affirmatively in a recent preprint of Olshanskii and Osin \cite{OOpre}, where certain Burnside groups are proved to be C*-simple. Using Theorem \ref{thm:c-star-simplicity}, we can now provide another solution to de la Harpe's problem.

Recall that for a fixed prime $p$, a Tarski monster group of order $p$ is a finitely generated group $G$ with the property that every nontrivial subgroup is cyclic of order $p$. Olshanskii \cite{Olsh82} proved the existence of Tarski monster groups for every prime $p > 10^{75}$, and further proved that these groups are non-amenable. This answered a question of von Neumann about the existence of non-amenable groups which do not contain non-abelian free subgroups.

We will now prove that the action of Tarski monster groups on the Furstenberg boundary is topologically free, and hence that they are C*-simple.

\begin{thm} \label{thm:tarski-monsters-top-free}
Let $G$ be a Tarski monster group, and let $\fb$ denote the Furstenberg boundary of $G$. Then the $G$-action on $\fb$ is topologically free.
\end{thm}
\begin{proof}
Suppose for the sake of contradiction that the $G$-action on $\fb$ is not topologically free. Then there is $s_0 \in G \setminus \{e\}$ such that the set $(\fb)^{s_0} = \{ x \in \fb \mid s_0 \in G_x\ \}$ has non-empty interior, where $G_x = \{s \in G \mid sx = x\}$. Hence there is a non-empty open set $V \subset (\fb)^{s_0}$. We claim $V$ is finite. To prove this, suppose to the contrary that $V$ is infinite. We will show this leads to a contradiction.

Fix some point $x \in V$. By the minimality of the $G$-action on $\fb$, there is $s_1 \in G \setminus \{e,s_0\}$ such that $s_1x \in V \setminus \{x\}$. In particular, $s_1x \in (\fb)^{s_0}$, and hence $s_0 \in G_{s_1x}$. Note that $G_{s_1x} = s_1G_x{s_1}^{-1}$. Since $s_0$ generates $G_x$, this implies $G_x = s_1G_x{s_1}^{-1}$, and hence that $s_1$ belongs to the normalizer $N_G(G_x)$. But $N_G(G_x)$ is a proper subgroup of $G$, and it follows that we must have $N_G(G_x) = G_x$. Hence $s_1 \in G_x$.

If $V \setminus \{x, s_1x\}$ is non-empty, then we can repeat this argument to find $s_2 \in G \setminus \{e, s_0, s_1\}$ such that $s_2x \in V \setminus \{x, s_1x\}$ and $s_2 \in G_x$. If $V$ is infinite, then continuing in this way we obtain  distinct elements $s_0,s_1,\ldots,s_n \in G_x$ and distinct points $x, s_1x, \ldots , s_nx \in V$ for each $n$. But since $G_x$ is finite, this is impossible, and it follows that $V$ must be finite.

Now since $V$ is non-empty, finite and open, it follows that $\fb$ has isolated points. But this contradicts Proposition \ref{prop:no-isolated-points}.
\end{proof}

The next result follows immediately from Theorem \ref{thm:tarski-monsters-top-free} and Theorem \ref{thm:c-star-simplicity}.

\begin{cor}
Tarski monster groups are C*-simple.
\end{cor}


\end{document}